\documentclass[12pt,  reqno]{amsart}

\usepackage{amsthm,amsmath}

\newtheorem{theorem}{Theorem}[section]
\newtheorem{lemma}{Lemma}[section]
\newtheorem{proposition}{Proposition}[section]
\newtheorem{definition}{Definition}[section]
\newtheorem{corollary}{Corollary}[section]

\newtheorem{remark}{Remark}

\numberwithin{equation}{section}
\theoremstyle{definition}
\theoremstyle{remark}

\begin{document}

\title{On the positivity of a quasi-local mass in general dimensions}
\author{Kwok-Kun Kwong}

\address{~School of Mathematical Sciences, Monash University, Victoria 3800, Australia.} \email{kwok-kun.kwong@monash.edu}

\thanks{Research partially supported by Australian Research Council Discovery Grant \#DP0987650 }

\renewcommand{\subjclassname}{\textup{2010} Mathematics Subject Classification}
\subjclass[2010]{Primary 53C20, Secondary 83C99}
\date{April, 2013}

\keywords{}

\begin{abstract}
In this paper, we obtain a positivity result of a quasi-local mass integral as proposed by Shi and Tam in general dimensions. The main argument is based on the monotonicity of a mass integral in a foliation of quasi-spherical metrics and a positive mass type theorem which was proved by Wang and Yau in the three dimensional case, and is shown here in higher dimensions using spinor methods.
 \end{abstract}
\maketitle\markboth{}{}

The well-known positive mass theorem states that for a complete asymptotically flat manifold $(M, g)$ such that $g$ behaves like Euclidean at infinity near each end and suppose its scalar curvature is non-negative, then its ADM mass \cite{ADM} of each end is non-negative. Moreover, if the ADM mass of one of the end is zero, then $(M,g)$ is actually the Euclidean space. The positive mass theorem was proved by Schoen and Yau
\cite{schoen1979proof, schoen1981proof} using minimal surface techniques. Later on, Witten \cite{witten} (see also \cite{parker-taubes, BTK86}) gave a simplified proof the positive mass theorem using the spinors. Since then the method of spinors has been adopted by many people to prove positive mass type theorems or some rigidity results, see for example \cite{andersson1998scalar,liu2003positivity,shi2002positive,WY, Wang}.

In particular, let us look at some results in this direction. Wang and Yau \cite{WY}  developed a quasi-local mass for a three dimensional manifold with boundary whose scalar curvature is bounded from below by some negative constant. Using spinors methods, they were able to prove that this mass is non-negative. Later on, Shi and Tam \cite{ShiTam2007} also proved a similar result in the three dimensional case, but with a simpler and more explicit definition of the mass. In this paper, we will show that the results of Shi-Tam and Wang-Yau also hold in higher dimensions.

More precisely, in \cite{ShiTam2007}, Shi and Tam proved the following:

\begin{theorem}(\cite{ShiTam2007} Theorem 3.1)\label{thm: 1}
Let $(\Omega, g)$ be a compact $3$-dimensional orientable manifold with smooth
boundary $\Sigma=\partial \Omega$, homeomorphic to a $2$-sphere. Assuming the following conditions:
\begin{enumerate}
\item
The scalar curvature $R$ of $(\Omega, g)$ satisfies $R\ge -6 k^2$ for some $k>0$,
\item
$\Sigma$ is a topological sphere with Gaussian curvature $K>-k^2$ and mean curvature $H>0$, so that $\Sigma$ can be isometrically embedded into $\mathbb{H}_{-k^2}^3$ (the hyperbolic space with curvature $-k^2$) with mean curvature $H_0$.
\end{enumerate}
Then there is a future  time-like vector-valued function $W$ on $\Sigma$ such that the vector
\begin{equation*}
\int_{\Sigma}(H_0-H) W \,d\Sigma\in \mathbb R^{3,1}
\end{equation*}
is future  non-spacelike. Here $W=(x_1, x_2, x_3, \alpha t)$ for some $\alpha>1$ depending only on the intrinsic geometry of $\Sigma$, where $X=(x_1, x_2, x_3, t)\in\mathbb{H}^3_{-k^2}\subset \mathbb{R}^{3, 1}$ is the position vector of the embedding of $\Sigma$.
\end{theorem}

In this paper, we will prove the analogous result in higher dimensions for spin manifolds (note that three dimensional orientable manifolds are spin). More precisely we will prove the following
\begin{theorem}(cf. Theorem \ref{thm: m<0})\label{thm: 2}
Let $n\geq 3$ and $(\Omega, g)$ be a compact spin $n$-manifold with smooth boundary $\Sigma$. Assuming the following conditions:
\begin{enumerate}
\item
The scalar curvature $R$ of $(\Omega, g)$ satisfies $R\ge-n(n-1) k^2$ for some $k>0$,
\item
$\Sigma$ is topologically a $(n-1)$-sphere with sectional curvature $K>-k^2$, mean curvature $H>0$ and $\Sigma$ can be isometrically embedded uniquely into $\mathbb{H}_{-k^2}^n$ with mean curvature $H_0$.
\end{enumerate}

Then there is a future  time-like vector-valued function $W$ on $\Sigma$ such that the vector
\begin{equation*}
\int_{\Sigma}(H_0-H) W \,d\Sigma\in \mathbb R^{n,1}
\end{equation*}
is future  non-spacelike.
 Here $W=(x_1, x_2, \cdots, x_n,  \alpha t)$  for some $\alpha>1$ depending only on the intrinsic geometry of $\Sigma$, where $X=(x_1, x_2, \cdots, x_n,  t)\in\mathbb{H}^n_{-k^2}\subset \mathbb{R}^{n, 1}$ is the position vector of the embedding of $\Sigma$.
\end{theorem}

The value of $\alpha $ in Theorem \ref{thm: 1} and \ref{thm: 2} will be given in Theorem \ref{thm: m<0}. To prove Theorem \ref{thm: 2}, there are two main ingredients. One is
the monotonicity of the mass integral under certain flow of the embedded surface in $\mathbb H^n_{-k^2}$, and the other is a positive mass type theorem (Theorem \ref{thm: PMT}). This positive mass theorem was originally proved by Wang and Yau \cite{WY} in the three dimensional case. Here we will give a proof in general dimension. In particular, the existence of the Killing spinor fields plays an important role in the proof. The crucial observations for the proof of the positive mass theorem are some identities involving Killing spinors on the hyperbolic space (Proposition \ref{prop: null}, \ref{prop: spinor}).

Let us also mention that it was conjectured (cf. \cite{ShiTam2007}) that the value of $\alpha$ in Theorem \ref{thm: 1} can be taken to be $1$. Indeed, it was proved by Tam and the author \cite{kwong2010limit} that in the three-dimensional case, if we define the quasilocal mass integral of $\Sigma$ to be
$
\int_\Sigma (H_0-H) X
$, then it has the desired limiting behavior, in the sense that
 this mass integral, when evaluated on coordinate spheres, will tend to the mass (cf. \cite{Wang}) of an asymptotically hyperbolic manifold (under suitable assumptions and normalization). It is also natural to ask if $\alpha$ in Theroem \ref{thm: 2} can be taken to be $1$.

This paper is organized as follows. In Section \ref{sec: prelim}, we will first state and prove some preliminary results, most of which are similar to those in \cite{ShiTam2007} and \cite{WY}. In Section \ref{sec: pmt}, we will state and prove a positive mass type theorem in general dimension. In particular, we will derive some results about Kiling spinors and the Dirac operator in this section for later use. In Section \ref{sec: thm}, we will give the proof of our main result.

{\sc Acknowledgments}:
The contents of this paper forms part the Ph.D. thesis of the author. He would like to give his sincere thanks to his supervisor Luen-Fai Tam for his constant encouragement and very careful guidance.
He would also like to thank Yuguang Shi for useful comments. Finally he would like to thank his wife Candy for her patience and love.

\section{Preliminaries}\label{sec: prelim}
In this section, we will state and prove some preliminary results which are similar to those in \cite{ShiTam2007} and \cite{WY}. The setup is as follows.

Let $(\Omega, g)$ be a compact $n$-dimensional manifold with smooth
boundary $\Sigma=\partial \Omega$, homeomorphic to a $(n-1)$-sphere. Suppose
the scalar curvature $R$ of $\Omega$ satisfies $R\geq -n(n-1)k^2$
for some $k>0$. Let $H$ be the mean curvature of $\Sigma$ with
respect to the outward normal. We assume $H$ is positive, the
sectional curvature of $\Sigma$ is greater than $-k^2$ and $\Sigma$
can be isometrically embedded uniquely into $\mathbb{H}_{-k^2}^n$, the
hyperbolic space of constant sectional curvature $-k^2$.
 We use the following hyperboloid model  for $\mathbb{H}_{-k^2}^n$:
\begin{equation}\label{eq: hyperb}
\mathbb{H}_{-k^2}^n =\left\{(x_1, \cdots, x_n, t)\in \mathbb R^{n, 1}\mid \sum_{i=1}^n x_i ^2-t^2=-\frac{1}{k^2}, t>0 \right\}
\end{equation}
where $\mathbb R^{n, 1}$ is the Minkowski space  with Lorentz metric $\displaystyle\sum_{i=1}^ndx_i^2-dt^2$. The position vector of $\mathbb{H}_{-k^2}^n$ in $\mathbb R^{n, 1}$ can be parametrized by
\begin{equation}\label{eq: spherical}
X=(x_1, \cdots, x_n, t)=\frac{1}{k}(\sinh (kr) Y, \cosh kr)
\end{equation}
where $Y\in \mathbb{S}^{n-1}$, the unit sphere in $\mathbb R^n$. Note that $r$ is the geodesic distance of a point from $o=(0, \cdots, 0, 1/k)\in \mathbb{H}_{-k^2}^n$. Without loss of generality  we can assume that $\Sigma_0$, the embedded image of $\Sigma$, encloses a region $\Omega_0$ which contains $o$.

Let $\Sigma _\rho$  be the level surface outside $\Sigma_0$ in $\mathbb{H}_{-k^2}^n$ with distance $\rho$ from $\Sigma_0$. Suppose $F: \Sigma \rightarrow \mathbb{H}_{-k^2}^n$ is the embedding with unit outward normal $N$, then $\Sigma_\rho$ as a subset of $\mathbb R^{n, 1}$ is given by (\cite{ShiTam2007} Equation (2.2))
\begin{equation}\label{eq: X rho}
X(p, \rho)=\cosh (k\rho) X(p, 0)+\frac{1}{k}\sinh (k \rho) N(p, 0).
\end{equation}
Here for simplicity, $(p, \rho)$ denotes a point $\Sigma_\rho$ which
lies on the geodesic perpendicular to $\Sigma_0$ starting from the
point $p\in \Sigma _0$ and $X(p, 0)=X(F(p))$.

On $\mathbb{H}^n_{-k^2}\setminus \Omega_0$, the hyperbolic metric can be written as
\begin{equation}\label{eq: g'}
g'=d\rho^2+g_\rho,
\end{equation}
where $g_\rho$ is the induced metric on $\Sigma_\rho$. As in \cite{WY}, we
can perturb the metric to form a new metric on $\mathbb{H}^n_{-k^2}\setminus \Omega_0 $
\begin{equation}\label{eq: g''}
g''=u^2d\rho^2+g_\rho
\end{equation}
(note that the induced metrics from $g'$ and $g''$
on $\Sigma_\rho$ are the same) with prescribed scalar curvature
$-n(n-1)k^2$, where $u$ satisfies (\cite{WY} Equation 2.10):
\begin{equation}\label{eq: pde}
\left\{
\begin{split}
2H_0 \frac{\partial u}{\partial\rho}&=2u^2\Delta _\rho u +(u-u^3)(R^\rho+n(n-1)k^2),\\
u(p, 0)&=\frac{H_0(p, 0)}{H(p)}.
\end{split}
\right.
\end{equation}
Here $\Delta _\rho$ is the Laplacian on $\Sigma_\rho$, $R^\rho$ is
the scalar curvature of $\Sigma_\rho$, $H_0(p, \rho)$ is the mean
curvature of $\Sigma_\rho$ in $(\mathbb{H}_{-k^2}^n, g')$ and $H(p)$
is the mean curvature of $\partial \Omega$ in $(\Omega, g)$. The
mean curvature of $\Sigma_\rho$ with respect to the new metric $g''$ is then
\begin{equation}\label{eq: H}
H(p, \rho)=\frac{H_0(p, \rho)}{u(p, \rho)}.
\end{equation}

We have the following estimates:
\begin{lemma}[cf. \cite{WY} p. 255-257]\label{lem: est}
\begin{enumerate}
\item\label{item: 1}
For all $\rho$, $e^{-2k \rho}g_\rho$ is uniformly equivalent to the standard metric on $\mathbb{S}^{n-1}$. Indeed, we can choose a coordinates around any $p\in\Sigma$ such that
$g_{ab}(p, \rho)=f\delta_{ab}$, where $f=\frac{\sinh^2(k (\mu_a+\rho))}{\sinh^2(k \mu_a)}, e^{2k \rho} $ or $\frac{\cosh^2(k (\mu_a+\rho))}{\cosh^2(k \mu_a)}$ and $\lambda_a(p,0)=k \coth (k \mu_a), k$ or $k \tanh (k \mu_a)$ is the initial principal curvature with respect to $g'$.
\item
Let $d\Sigma_\rho$ denotes the volume element of $\Sigma_\rho$, then $e^{-(n-1)k\rho}d\Sigma_\rho$ is uniformly equivalent to the volume element $d\mathbb{S}^{n-1}$ of $\mathbb{S}^{n-1}.$
\item
The principal curvatures of $\Sigma_\rho$ with respect to $g'$ is of order $\lambda_a(p, \rho)=k(1+O(e^{-2k\rho}))$, and therefore $H_0=(n-1)k+O(e^{-2k\rho})$.
\item
$|u-1|\leq C e^{-nk\rho}$ for some $C>0$ independent of $\rho$.
\end{enumerate}
\end{lemma}
We also have the following long time existence result:
\begin{proposition}[cf. \cite{WY} Theorem 2.1]\label{prop: u}
\begin{enumerate}
\item
The solution $u$ of \eqref{eq: pde} exists for all time and $v=\displaystyle\lim_{\rho\rightarrow \infty} e^{nk\rho} (u-1)$ exists as a smooth function on $\Sigma$.
\item
$g''=u^2d\rho^2+ g_\rho$ is asymptotically hyperbolic \cite{andersson1998scalar} on $M=\mathbb{H}^n_{-k^2}\setminus \Omega_0$ with scalar curvature $-n(n-1)k^2$.
\item
Let $A:(TM, g')\rightarrow(TM, g'')$ be the Gauge transformation defined by $A \frac{\partial}{\partial \rho}=\frac{1}{u}\frac{\partial}{\partial\rho}$ and $AV=V$ for any vector $V\in T\Sigma_\rho$, then $|A-Id|_{g'}=O(e^{-nk\rho})$ and $|\nabla'A|_{g'}=O(e^{-nk\rho})$.
\end{enumerate}
\end{proposition}
Since the proofs of the above two results are exactly the same as in \cite{WY} except some obvious modification, we omit them here.
\begin{lemma}(cf. \cite{ShiTam2007} Lemma 3.4)\label{lem: X}
On $\mathbb{H}^n_{-k^2}\setminus \Omega_0$,
\begin{equation*}
H_0 \frac{\partial X}{\partial\rho}+\Delta_\rho X-(n-1)k^2 X=0.
\end{equation*}
\end{lemma}
\begin{proof}
First of all it is easy to see that $\Delta _{\mathbb{H}_{-k^2}^n}X=nk^2 X$.
On the other hand, under the foliation by $\Sigma_\rho$, the $\Delta_{\mathbb{H}_{-k^2}^n}$ is given by
$\Delta_{\mathbb{H}_{-k^2}^n} =\frac{\partial^2}{\partial \rho^2}+H_0 \frac{\partial }{\partial\rho}+\Delta_\rho$,
where $\Delta _\rho$ is the Laplacian on $\Sigma_\rho$.
So using \eqref{eq: X rho},
$$
nk^2X=\frac{\partial ^2}{\partial\rho^2}X+H_0 \frac{\partial}{\partial\rho}X+\Delta_\rho X
=k^2 X +H_0 \frac{\partial}{\partial\rho}X+\Delta_\rho X.
$$
\end{proof}
 Let $B_0(R_1)$ and $B_0(R_2)$ be geodesic balls in $\mathbb{H}^n_{-k^2}$ such that $B_0(R_1)\subset D\subset B_0(R_2)$. We define
 $W=(x_1, x_2, \cdots, x_n, \alpha t)$ with
$$\alpha=\coth kR_1+\frac{1}{\sinh k R_1}\left(\frac{\sinh^2 kR_2}{\sinh ^2 k R_1}-1 \right)^\frac{1}{2},$$
where $X=(x_1, x_2, \cdots, x_n,  t)$ is the position vector of $\Sigma_\rho$ in $\mathbb R^{n, 1}$.

Clearly we also have
\begin{lemma}\label{lem: W}
On $\mathbb{H}^n_{-k^2}\setminus \Omega_0$,
$
H_0 \frac{\partial W}{\partial\rho}+\Delta_\rho W-(n-1)k^2 W=0.
$
\end{lemma}

\begin{lemma}(cf. \cite{ShiTam2007} Equation 3.8)\label{lem: der}
On $\mathbb{H}^n_{-k^2}\setminus \Omega_0$,
\begin{equation*}
\begin{split}
&\frac{d}{d\rho}\left(\int_{\Sigma_\rho}(H_0-H)X d\Sigma_\rho \right )\\
=&-\int_{\Sigma_\rho}u^{-1}(u-1)^2  \left(\left(R^\rho+(n-1)(n-2)k^2 \right)\frac{X}{2} + H_0\frac{\partial X}{\partial\rho}\right) d\Sigma_\rho.
\end{split}
\end{equation*}
\end{lemma}

\begin{proof}
By \eqref{eq: pde} and the divergence theorem,
\begin{equation}\label{eq: d}
\begin{split}
&\frac{d}{d\rho}\left (\int_{\Sigma_\rho}(H_0-H)X d\Sigma_\rho \right)\\
=&\frac{d}{d\rho}\left (\int_{\Sigma_\rho}H_0(1-u^{-1})X d\Sigma_\rho \right) \\
=&\int_{\Sigma_\rho} (\frac{\partial H_0}{\partial\rho} (1-u^{-1})X +
H_0 u^{-2}\frac{\partial u}{\partial \rho} X +
H_0 (1-u^{-1})\frac{\partial X}{\partial\rho} \\
&+H_0^2 (1-u^{-1})X  ) d\Sigma_\rho\\
=&\int_{\Sigma_\rho} ( (\frac{\partial H_0}{\partial\rho}+H_0^2) (1-u^{-1})X \\
&+\left(\Delta_\rho u+\frac{1}{2}(u^{-1}-u)(R^\rho+n(n-1)k^2)\right)X+H_0 (1-u^{-1})\frac{\partial X}{\partial\rho} ) d\Sigma_\rho\\
=&\int_{\Sigma_\rho} ( (\frac{\partial H_0}{\partial\rho}+H_0^2) (1-u^{-1})X
+\frac{1}{2}(u^{-1}-u)(R^\rho+n(n-1)k^2)X\\
&+H_0 (1-u^{-1})\frac{\partial X}{\partial\rho}+(u-1)\Delta_\rho X )d\Sigma_\rho\\
=&\int_{\Sigma_\rho} (\textrm{I}+\textrm{II}+\textrm{III}+\textrm{IV}) d\Sigma_\rho
\end{split}
\end{equation}
where we have used \eqref{eq: pde} in line 4 and divergence theorem
in line 5. The Gauss equation gives
\begin{equation}\label{eq: gauss}
R^\rho=-(n-1)(n-2)k^2+H_0^2-|A|^2
\end{equation}
where $A$ is the second fundamental form of $\Sigma_\rho$ with
respect to the hyperbolic metric $g'$. By the evolution equation of $H_0$ (\cite{WY} Equation (2.4))
and the Gauss equation \eqref{eq: gauss},
\begin{equation*}
\begin{split}
\frac{\partial H_0}{\partial\rho}
=-|A|^2+(n-1)k^2
=R^\rho+(n-1)^2k^2-H_0^2.
\end{split}
\end{equation*}
So $\textrm{I}=(R^\rho +(n-1)^2 k^2)(1-u^{-1})X$.\\
Direct calculation gives
\begin{equation*}
\begin{split}
&(R^\rho+(n-1)^2k^2)(1-u^{-1})+\frac{1}{2}(u^{-1}-u)(R^\rho+n(n-1)k^2)\\
=&-\frac{1}{2}u^{-1}(u-1)^2 (R^\rho+(n-1)(n-2)k^2)-(n-1)(u-1)k^2.
\end{split}
\end{equation*}
So we have
\begin{equation*}
\textrm{I}+\textrm{II}=(-\frac{1}{2}u^{-1}(u-1)^2 \left (R^\rho+(n-1)(n-2)k^2 \right)-(n-1)(u-1)k^2 ) X.
\end{equation*}
By lemma \ref{lem: X}, $\Delta_\rho X-(n-1)k^2 X=-H_0 \frac{\partial X}{\partial\rho}$.
Therefore
\begin{equation*}
\begin{split}
 & \textrm{I}+\textrm{II}+\textrm{III}+\textrm{IV}\\
=&-\frac{1}{2}u^{-1}(u-1)^2 \left(R^\rho+(n-1)(n-2)k^2 \right)X  \\ &+(u-1)\left(\frac{H_0}{u}\frac{\partial X}{\partial\rho}+\Delta _\rho X-(n-1)k^2 X \right)\\
=&-\frac{1}{2}u^{-1}(u-1)^2 \left(R^\rho+(n-1)(n-2)k^2 \right)X  +(u-1)(u^{-1}-1)H_0\frac{\partial X}{\partial \rho}\\
=&-u^{-1}(u-1)^2  \left(\left(R^\rho+(n-1)(n-2)k^2 \right)\frac{X}{2} + H_0\frac{\partial X}{\partial\rho}\right).
\end{split}
\end{equation*}
This together with \eqref{eq: d} gives the result.
\end{proof}
\section{A positive mass theorem}\label{sec: pmt}
We will need the following positive mass type theorem (cf. \cite{WY} Theorem 6.1 and Corollary 6.3) which was proved by Wang and Yau when $n=3$.
\begin{theorem}[Wang-Yau]\label{thm: PMT}
Let $n\geq 3$
and $(\Omega, g)$ is a n-dimensional compact spin manifold with nonempty smooth boundary which is a topological sphere. Suppose the scalar curvature $R$ of $\Omega$ satisfies $R\geq -n(n-1) k^2$, the sectional curvature of its boundary $\Sigma$ satisfies $K>-k^2$, the mean curvature of the boundary with respect to outward unit normal is positive, and $\Sigma$ can be isometrically embedded uniquely into $\mathbb{H}^n_{-k^2}$ in $\mathbb{R}^{n, 1}$. Then
$$\lim_{\rho\rightarrow\infty} \int_{\Sigma_\rho}(H_0-H)X\cdot \zeta\leq 0$$
for any future-directed null vector $\zeta$ in $\mathbb{R}^{n, 1}$. Here $H_0, H$ are functions in $(p, \rho)$ as in \eqref{eq: H} and $X\in\mathbb{H}^n_{-k^2}\subset \mathbb{R}^{n, 1}$ is the position vector of the isometric embedding of $\Sigma$.\\
In other words, $\displaystyle\lim_{\rho\rightarrow\infty} \int_{\Sigma_\rho}(H_0-H)X \,d\Sigma_\rho$ is a future non-spacelike vector.
\end{theorem}
As a corollary,
\begin{corollary}\label{corollary}
With the same assumptions as in Theorem \ref{thm: PMT},
$$\lim_{\rho\rightarrow \infty}\int_{\Sigma_\rho}(H_0-H)\cosh kr \,d\Sigma_\rho\geq 0$$
where $r$ is defined in \eqref{eq: spherical}.
\end{corollary}
\subsection{Killing spinors on $(\mathbb{H}^n_{-k^2}, g')$}
The proof of Theorem \ref{thm: PMT} requires the existence of Killing spinor fields (i.e. a section of the spinor bundle $S(\mathbb{H}^n_{-k^2}, g')$ satisfying the Killing equation \eqref{eq: killing}, see \cite{lawson1989spin}) on the hyperbolic space. A Killing spinor  $\phi'$ on $(\mathbb{H}^n_{-k^2}, g')$ satisfies the equation
\begin{equation}\label{eq: killing}
\nabla'_V \phi'+\frac{\sqrt{-1}}{2}k c'(V)\phi'=0 \text{ \quad for any tangent vector }V
\end{equation}
where $c'(V)$ is the Clifford multiplication by $V$ and $\nabla'$ is the spin connection  (with respect to the hyperbolic metric $g'$). The Killing spinors on hyperbolic spaces were studied by Baum \cite{B}, who proved that on $\mathbb{H}^{n}_{-k^2}$, the set of all Killing spinors is parametrized by $a\in \mathbb{C}^{2^m}$, $m=\lfloor\frac{n}{2}\rfloor$ (integer part). The following two propositions are crucial.
\begin{proposition}\label{prop: spinor}
Let $\phi_{a, 0}'$ be the Killing spinor on $(\mathbb{H}^n_{-k^2},g')$, corresponding to $a\in \mathbb{C}^{2^m}$, $m=\lfloor \frac{n}{2}\rfloor$, then
$$|\phi_{a, 0}'|_{g'}^2=-2kX\cdot \zeta_a$$
where $\cdot$ denotes the Lorentz inner product in $\mathbb{R}^{n, 1}$ and
\begin{equation}\label{eq: zeta}
\zeta_a=\sum_{j=1}^n\langle \sqrt{-1} c(e_j)a, a\rangle e_j- \langle  \sqrt{-1} c(e_0)a, a\rangle e_0.
\end{equation}
Here $\langle \cdot,\cdot \rangle$ is the inner product in $\mathbb{C}^{2^m}$, $c(e_j)$ denotes the Clifford multiplication by the Clifford matrices (as defined in \cite{B} p.206) for the orthonormal basis $\frac{\partial}{\partial t}=e_0,\frac{\partial }{\partial x_j}=e_j$ in $\mathbb{R}^{n, 1}$ ($1\leq j\leq n$) and $c(e_0)$ is defined to be $\sqrt{-1}I$, where $I$ is the identity matrix.
\end{proposition}
\begin{proof}
Let $k=1$ for simplicity. Baum (\cite[Theorem 1]{B}, $\mu=-\frac{k}{2}$) proved that in the ball model for $\mathbb {H}^n$, the Killing spinor  can be expressed as (note that the spinor bundle is trivial)
$$\phi=\phi_{a, 0}'(x)=\sqrt{\frac{2}{1-|x|^2}}(a- \sqrt{-1} c(x)a)$$
where $\displaystyle c(x)a=\sum_{j=1}^n x_j c(e_j)a$ for $x=(x_1,\cdots, x_n)$. It is easily computed that
$$|\phi|^2=2\left(\frac{1+|x|^2}{1-|x|^2}|a|^2-\frac{2}{1-|x|^2}\langle  \sqrt{-1} c(x)a, a \rangle\right).$$
The change of coordinates from the ball model to the hyperboloid model is given by
$$X=(\frac{2x}{1-|x|^2},\frac{1+|x|^2}{1-|x|^2})\in \mathbb{R}^{n,1}.$$
So
\begin{equation*}
\begin{split}
-2X\cdot \zeta_a
&=-4\sum_{j=1}^n \frac{x_j}{1-|x|^2}\langle  \sqrt{-1} c(e_j)a, a\rangle + 2\frac{1+|x|^2}{1-|x|^2}|a|^2\\
&=2\left(\frac{1+|x|^2}{1-|x|^2}|a|^2 -\frac{2}{1-|x|^2}\langle  \sqrt{-1} c(x)a, a \rangle\right)\\
&=|\phi|^2.
\end{split}
\end{equation*}
\end{proof}
\begin{proposition}\label{prop: null}
For every null vector $\zeta\in \mathbb{R}^{n,1}$ ($n\geq 2$), $\zeta=\zeta_a$ for some $a\in \mathbb{C}^{2^m}$, where $m=\lfloor\frac{n}{2}\rfloor$ and $\zeta_a$ is defined in \eqref{eq: zeta}.
\end{proposition}
\begin{proof}
Define $\displaystyle\eta_a=\sum_{j=1}^{n}\langle  \sqrt{-1}  c(e_j)a, a\rangle e_j$. As $\displaystyle \zeta_a=\sum_{j=1}^{n}\langle  \sqrt{-1}  c(e_j)a, a\rangle e_j + |a|^2 e_0$, it suffices to prove that for any $X\in \mathbb{S}^{n-1}\subset \mathbb{R}^{n}$, there exists $a\in \mathbb{C}^{2^m}$ with $|a|=1$ such that $\eta_a=X$. This can be proved in a similar way as in \cite{Wang} p.285-286. Here we use a different proof which is more explicit. We divide into two cases: (i) $n$ is odd and (ii) $n$ is even.

(i) For the odd case where $n=2m+1$, we apply induction on $m$. When $2m+1=3$, this is done in \cite{WY} (p.17). We state it here for later use. The three Clifford matrices for $n=3$ are $g_1, g_2$ and $ \sqrt{-1} T$ (see \cite[p. 206]{B}), where
\begin{equation*}
\begin{split}
g_1=
\begin{pmatrix}
 \sqrt{-1}  & 0\\
0 & - \sqrt{-1}
\end{pmatrix},
g_2=
\begin{pmatrix}
0 &  \sqrt{-1} \\
 \sqrt{-1}  & 0
\end{pmatrix},
T=
\begin{pmatrix}
0 & - \sqrt{-1} \\
 \sqrt{-1}  & 0
\end{pmatrix}.
\end{split}
\end{equation*}
So for any $\vec z\in \mathbb{S}^2$, there exists $a\in\mathbb{C}^2$ with $|a|=1$ such that
\begin{equation}\label{eq: z}
\vec z=(\langle  \sqrt{-1}  g_1 (a), a\rangle, \langle  \sqrt{-1}  g_2 (a), a\rangle, \langle -T (a), a\rangle).
\end{equation}
Assume the result is true for $n=2m-1$, and denote the Clifford matrices in dimension $2m-1$ simply by $\{c_j\}_{j=1}^{2m-1}$. Let $\{d_j\}_{j=1}^{2m+1}$ be the Clifford matrices in dimension $2m+1$, as defined in \cite[p. 206 Equation (2)]{B}. Then it is easily seen that
\begin{equation}\label{eq: ef}
\begin{split}
\begin{cases}
d_j=I\otimes c_j \quad \text{for $j=1, \cdots, 2m-2$,  $I$ is the $2\times2$ identity matrix},\\
d_{2m-1}=- \sqrt{-1}  g_1 \otimes c_{2m-1},\\
d_{2m}=- \sqrt{-1}  g_2 \otimes c_{2m-1},\\
d_{2m+1}=T\otimes c_{2m-1}.
\end{cases}
\end{split}
\end{equation}
Now let $X\in \mathbb{S}^{2m}$, then $X=(y_1, y_2, \cdots, y_{2m-1}\vec z)$ for some $y=(y_1, \cdots, y_{2m-1})\in \mathbb{S}^{2m-2}$ and $\vec z\in \mathbb{S}^2$. By induction assumption, there exists $b\in \mathbb{C}^{2^{m-1}}$ with $|b|=1$ such that
$$y=(\langle  \sqrt{-1}  c_1 (b), b\rangle, \cdots, \langle  \sqrt{-1}  c_{2m-1} (b), b\rangle)$$ and by \eqref{eq: z}, there exists $a\in \mathbb{C}^2$ with $|a|=1$ such that $$-\vec z=(\langle  \sqrt{-1}  g_1 (a), a\rangle, \langle  \sqrt{-1}  g_2 (a), a\rangle, \langle -T(a), a\rangle).$$
Combining these with \eqref{eq: ef}, it is easily seen that
$\eta_{a\otimes b}=X$.

(ii) For the even case, we also apply induction on $m$. When $n=2$, the two Clifford matrices are $g_1$ and $g_2$ and $\eta_a =(-|a_1|^2+|a_2|^2, -a_1 \overline a_2- a_2 \overline a_1)$ where $a=(a_1, a_2)\in \mathbb{C}^2$.
For $X=(\cos \theta, \sin \theta)\in \mathbb{S}^1$, just take $a=(-\sin \frac{\theta}{2}, \cos \frac{\theta}{2})$ so that $\eta_a=X$.

Assume the result is true for $n=2m$ and denote $\{c_j\}_{j=1}^{2m}$ to be the corresponding Clifford matrices as defined in \cite[p. 206 Equation (1)]{B}. Let $\{d_j\}_{j=1}^{2m+2}$ be the Clifford matrices for $n=2m+2$. Then it is easily seen that
\begin{equation}\label{eq: ef even}
\begin{split}
\begin{cases}
d_1=I\otimes g_1 \quad\text{where $I$ is the $2^m\times 2^m$ identity matrix},\\
d_2=I\otimes g_2 \quad\text{where $I$ is the $2^m\times 2^m$ identity matrix},\\
d_{j+2}=c_j \otimes T\quad \text{for $j=1, \cdots, 2m$.}
\end{cases}
\end{split}
\end{equation}
Now let $X\in \mathbb{S}^{2m+1}$, then $X=(z_1, z_2, z_3 \vec y)$ for some $(z_1, z_2, z_3)\in \mathbb{S}^2$ and $\vec y \in \mathbb{S}^{2m-1}$. By \eqref{eq: z}, there exists $b\in \mathbb{C}^2$ with $|b|=1$ such that
$$(z_1, z_2, -z_3)=(\langle  \sqrt{-1}  g_1 (b), b\rangle, \langle  \sqrt{-1}  g_2 (b), b\rangle, \langle -T(b), b\rangle)$$
and by induction assumption, there exists $a\in \mathbb{C}^{2^{m-1}}$ with $|a|=1$ such that
$$\vec y=(\langle  \sqrt{-1}  c_1(a),a\rangle , \cdots, \langle  \sqrt{-1}  c_{2m}(a),a \rangle ).$$
Combining these with \eqref{eq: ef even}, it is easily seen that $\eta_{a\otimes b}=X.$
\end{proof}
\subsection{The hypersurface Dirac operator}
In this subsection, we will give some general results for the hypersurface Dirac operator. Most of the materials in this section can be found, for example, in \cite{Hij}.

Recall that on the spinor bundle $S(M^n)$ over a spin manifold $(M,g)$, the Dirac operator  $D$ is defined to be
$$D\psi=\sum_{i=1}^n c_M(e_i)\nabla^M_{e_i}\psi$$
for any spinor $\psi\in \Gamma(S(M))$, where $\{e_i\}_{i=1}^n$ is a local orthonormal frame on $M$, $c_M$ is the Clifford multiplication and $\nabla^M$ is the spin connection on $S(M)$. The local formula for $\nabla^M$ is given by \cite[Theorem 4.14]{lawson1989spin}
$$\nabla^M_{e_i}\psi =e_i(\psi)+\frac{1}{2}\sum_{j<k}^n \langle \nabla_{e_i}e_j, e_k\rangle c_M(e_j)c_M(e_k)\psi,$$
where $\{e_i\}_{i=1}^n$ are orthonormal frames on $M$. For simplicity, let us write $c$ for $c_M$ and $\nabla$ for $\nabla^M$.

Now, for a spin manifold $M$, if $\Sigma\subset M$ is an oriented smooth hypersurface, then $M$ induces a natural spin structure on $\Sigma$, compatible with the induced orientation from $M$.

 We let $S:=S( M^n)|_{\Sigma}$, the restriction of the spinor bundle of $M$ to $\Sigma$. Then it can be shown that $S=S(\Sigma)$ when $n$ is odd and $S=S(\Sigma)\oplus S(\Sigma)$ when $n$ is even. We will work on $S$ (instead of $S(\Sigma)$).
\begin{definition}

 We define the hypersurface spin connection $\nabla ^S$, the hypersurface Clifford multiplication $c_S$ and the hypersurface Dirac operator  $D^S$ on $S$ by
\begin{equation*}
\begin{split}
\nabla^S_X \psi &=\nabla_X \psi+\frac{1}{2}c(\nu)c(B(X))\psi, \\
c_S(X)&=-c(\nu)c(X),\\
D^S\psi&=\sum_{a=1}^{n-1}c_S(e_a)\nabla^S_{e_a}\psi.
\end{split}
\end{equation*}
 where $\nu$ is a fixed unit normal (outward if this makes sense) and $B$ is the shape operator on $\Sigma $, i.e. $B(X)=-\nabla_X\nu$.
 \end{definition}
 In local formula, for $\{e_a\}_{a=1}^{n-1}$ orthonormal on $\Sigma$ and $e_n=\nu$ be the unit outward normal,
\begin{equation}\label{eq: nabla}
\begin{split}
\nabla_{e_a}^S\psi
&=\nabla_{e_a}\psi+\frac{1}{2} \sum_{b=1}^{n-1}h_{ab} c(e_b)c(e_n)\psi.
\end{split}
\end{equation}
(It can be verified that $\nabla^S=\nabla^{\Sigma }\oplus \nabla^{\Sigma }$ and $c_S=c_{\Sigma }\oplus -c_{\Sigma }$ when $n$ is even. )\\
\begin{definition}
We define the Killing spin connection  $\widehat \nabla$, Killing Dirac  operator $\widehat D$ and the Killing boundary operator  $\widehat B$ respectively by
\begin{equation}\label{eq: operators}
\begin{split}
\widehat \nabla_V\psi&=\nabla _V\psi +\frac{\sqrt{-1}}{2}k c(V)\psi,\\
\widehat D \psi&=\sum_{i=1}^n c(e_i)\widehat \nabla_{e_i}\psi,\\
\widehat B \psi&=\sum_{a=1}^{n-1}c(e_n)c(e_a)\widehat \nabla_{e_a}\psi
= \widehat \nabla _{e_n}\psi+ \widehat D\psi.
\end{split}
\end{equation}
\end{definition}
Actually $\widehat B$ is the boundary operator for the Lichnerowicz type formula   (\cite{WY} Equation 3.2): for any bounded region $U$ with smooth boundary, we have
\begin{equation}\label{eq: Lichnerowicz}
\int_{U}\left(\langle\widehat \nabla \psi,\widehat \nabla \varphi\rangle +\frac{1}{4} (R+n(n-1)) \langle\psi, \varphi\rangle - \langle\widehat D \psi, \widehat D \varphi\rangle\right)
=\int_{\partial U}\langle \psi, \widehat B\varphi\rangle.
\end{equation}
where $R$ is the scalar curvature.

From now on until the end of this section, the indices $a, b, c$ run from $1$ to $n-1$ and $i, j, k$ run from $1$ to $n$. Repeated indices will be summed over.

\begin{proposition}\label{prop: bound}
Let $\psi$ be a spinor on $M$ and $H$ is the mean curvature of $\Sigma\subset M$. Then on $\Sigma$,
\begin{equation*}
\widehat B \psi=-D^S \psi-\frac{H}{2}\psi-\frac{\sqrt{-1}}{2}k (n-1) c(e_n)\psi.
\end{equation*}

\end{proposition}

\begin{proof}
We have $\widehat B \psi=c(e_n)c(e_a)\widehat \nabla_{e_a}\psi$.
Using \eqref{eq: nabla},
$$c(e_n)c(e_a)\widehat \nabla_{e_a}\psi =c(e_n)c(e_a)\nabla^S _{e_a}\psi+\frac{1}{2}h_{ab}c(e_a)c(e_b)-\frac{\sqrt{-1}}{2}k (n-1) c(e_n)\psi.$$
We have $c_S(e_a)=c(e_a)c(e_n)$, so $c(e_n)c(e_a)\nabla^S_{e_a}=-D^S$. Also, $h_{ab}c(e_a)c( e_b)=-H$. The result follows.
\end{proof}

Let us now return to the hyperbolic space. More precisely, define $M=\mathbb{H}^n_{-k^2}\setminus \Omega_0$. Let $A:(TM, g')\rightarrow(TM, g'')$ be the Gauge transformation defined by $A \frac{\partial}{\partial \rho}=\frac{1}{u}\frac{\partial}{\partial\rho}$ ($u$ as defined in \eqref{eq: pde}) and $AV=V$ for any vector $V$ tangential to $\Sigma_\rho$. $A$ can be lifted to the spinor bundles as an isometry \cite{andersson1998scalar}, i.e. $A: S(M, g')\rightarrow S(M, g'')$. Also,
$$A(c'(X)\psi)=c''(AX)A\psi$$
where $c'$ (resp. $c''$) denotes the Clifford multiplication associated to $g'$ (resp. $g''$). We will also denote by $e_n''$ (resp. $e_n'$) to denote the unit outward normal of $\Sigma_\rho$ with respect to $g''$ (resp. $g'$).
\begin{proposition}\label{prop: D^S}
  Let $\phi'_0$ be a Killing spinor with respect to $\nabla'$ on $M=   \mathbb{H}^n_{-k^2}\setminus \Omega_0$ and $\phi_0=A\phi'_0$. Let $D^S$ be the hypersurface Dirac operator on $\Sigma_\rho$ with respect to $(M, g'')$ as defined in \eqref{eq: operators}. Then on $\Sigma_\rho$,
$$-D^S \phi_0=\frac{H_0}{2} \phi_0 +\frac{\sqrt{-1}}{2}k(n-1)c''(e_n'')\phi_0.$$
(Recall that $H_0$ is the mean curvature of $\Sigma_\rho$ with respect to $g'$. )
\end{proposition}
\begin{proof}
(The idea is the same as \cite{WY} Proposition 2.4 (modulo a minor misprint there), except we have to replace $\nabla^{\Sigma_\rho}$ by $\nabla^S$ etc. )

For $\overline \nabla_{e_a}\psi:=A\nabla_{e_a}(A^{-1}\psi)$, we have
\begin{equation}\label{eq: phi0}
\overline\nabla_{e_a}\phi_0=-\frac{\sqrt{-1}}{2}kc''(e_a)\phi_0
\end{equation}

Consider
\begin{equation}\label{eq: bar nabla}
\begin{split}
\overline \nabla_{e_a}\psi
=&e_a(\psi)+\frac{1}{2}\sum_{b<c}^{n-1}g''(\overline \nabla_{e_a} e_b , e_c)c''(e_b)c''(e_c)\psi\\
&+\frac{1}{2}\sum_{b=1}^{n-1}g''(\overline \nabla_{e_a}e_b, e_n'')c''(e_b)c''(e_n'')\psi.
\end{split}
\end{equation}
Note that $g''(\overline \nabla_{e_a} e_b , e_c)=g''(A \nabla'_{e_a} A^{-1}e_b , e_c)=g''(A \nabla'_{e_a} e_b , Ae_c)=g'( \nabla'_{e_a} e_b , e_c)=g''( \nabla''_{e_a} e_b , e_c)$ (as $g'|_{\Sigma_\rho}=g''|_{\Sigma_\rho}$). Also, $g''(\overline \nabla_{e_a}e_b, e_n'')=g''(A \nabla'_{e_a}A^{-1}e_b, Ae_n')=g'( \nabla'_{e_a}e_b, e_n')=-h_{ab}^0$.
So \eqref{eq: bar nabla} becomes
\begin{equation}\label{eq: bar2}
\begin{split}
\overline \nabla_{e_a}\psi
&=e_a(\psi)+\frac{1}{2}\sum_{b<c}^{}g''(\nabla''_{e_a} e_b , e_c)c''(e_b)c''(e_c)\psi
-\frac{1}{2}h_{ab}^0 c''(e_b)c''(e_n'')\psi\\
&=\nabla^S_{e_a}\psi
-\frac{1}{2}h_{ab}^0 c''(e_b)c''(e_n'')\psi \text{\quad by \eqref{eq: nabla}}.
\end{split}
\end{equation}
Note that by definition of $D^S$ and $c_S$,
$$D^S \psi=c_S(e_a)\nabla^S_{e_a}\psi=-c''(e_n'')c''(e_a)\nabla^S_{e_a}\psi.$$
So using \eqref{eq: bar2} and \eqref{eq: phi0},
\begin{equation*}
\begin{split}
D^S\phi_0&=-c''(e_n'')c''(e_a)\left(\overline \nabla_{e_a}\phi_0+ \frac{1}{2}h_{ab}^0 c''(e_b)c''(e_n'')\phi_0 \right)\\
&=-c''(e_n'')c''(e_a)\left(-\frac{\sqrt{-1}}{2}kc''(e_a)\phi_0+ \frac{1}{2}h_{ab}^0 c''(e_b)c''(e_n'')\phi_0 \right)\\
&=-\frac{\sqrt{-1}}{2}k(n-1)c''(e_n'')\phi_0-\frac{H_0}{2}\phi_0.
\end{split}
\end{equation*}
\end{proof}

\begin{proposition}\label{prop: lim mass}
With the assumptions in Theorem \ref{thm: PMT}, let $\phi_{a,0}'$ be a Killing spinor with respect to $g'$ and $\phi_0=A\phi_{0}'$ on $M$. Then the limit $\displaystyle \lim_{\rho \rightarrow \infty}\int_{\Sigma_\rho}(H_0-H) |\phi_{0}|_{g''}^2d\Sigma_\rho$ exists.
\end{proposition}
\begin{proof}
  $\phi_0'=\phi_{a,0}'$ as in Proposition \ref{prop: spinor}. By Proposition \ref{prop: spinor}, $|\phi_{a,0}|_{g''}^2=-2kX\cdot \zeta_a$. By \eqref{eq: X rho}, $e^{-k\rho}X(p,\rho)\rightarrow \gamma(p)= X(p,0)+\frac{1}{k}N(p,0)$.

  Also $e^{nk\rho}(H_0-H)=H_0e^{nk\rho}(1-u^{-1})\rightarrow (n-1)kv$ as given by Lemma \ref{lem: est} and Proposition \ref{prop: u}. By Lemma \ref{lem: est} again, $e^{-(n-1)k\rho}d\Sigma_\rho$ tends to a measure $d\mu$ on $\Sigma$, induced by the metric $g_\infty=\displaystyle\lim_{\rho \rightarrow \infty}e^{-2\rho}g_\rho$.
  All the above limits are uniform in $\rho$. Thus we have
  \begin{equation*}
\begin{split}
&\int_{\Sigma_\rho}(H_0-H) |\phi_{0}|_{g''}^2d\Sigma_\rho\\
=&-2k \int_{\Sigma}H_0(e^{nk\rho}(1-u^{-1}))(e^{-k\rho}X(p,\rho)\cdot \zeta_a)e^{-(n-1)k\rho}d\Sigma_\rho\\
\rightarrow & -2(n-1)k^2 \int_{\Sigma}v(\gamma\cdot \zeta_a)d\mu.
\end{split}
\end{equation*}
\end{proof}

\subsection{Proof of Theorem \ref{thm: PMT}}
Following the ideas in \cite{WY} Theorem 6.1 and Corollary 6.3, we now give the proof of Theorem \ref{thm: PMT}.
\begin{proof}[Proof of Theorem \ref{thm: PMT}]

 Define $g''=u^2d\rho^2+g_\rho$ on $M=\mathbb{H}^n_{-k^2}\setminus \Omega_0$ as in \eqref{eq: g''}, with $u$ satisfying \eqref{eq: pde}. Let $\widetilde g$ be the metric defined on $\widetilde M=M\cup_F \Omega$ such that $\widetilde g=g$ on $\Omega$ and $\widetilde g=g''$ on $M$, where $F$ is the embedding of $\Omega$ into $\mathbb{H}_{-k^2}^n$. Note that $\tilde g$ is Lipschitz near $\partial \Omega$, i.e. there is a smooth coordinates around $\partial \Omega$ such that the coefficients $\widetilde g_{ij}$ are Lipschitz.

 Let $\widehat \nabla_V=\widetilde \nabla _V +\frac{\sqrt{-1}}{2} k \widetilde c(V)$ and $\widehat D=\widetilde c(e_i) \widehat \nabla _{e_i}$ be the Killing connection and Killing-Dirac operator associated with $\widetilde g$ respectively. (All inner products and norms in this proof are taken with respect to $\tilde g$ unless otherwise stated. )

Let $\phi_0'$ be a Killing spinor on $\mathbb{H}^n_{-k^2}$ and $\phi_0=A\phi_0'$ on $M$, we claim that there is a (Killing harmonic) spinor $\phi$ with $\widehat D\phi =0$ on $\widetilde M$ such that
\begin{equation}\label{eq: killing harmonic}
0\leq \lim_{m\rightarrow\infty} \int_{\Sigma_{\rho_m}} \langle \phi, \widehat B\phi\rangle
=\lim_{m\rightarrow\infty} \int_{\Sigma_{\rho_m}} \langle \phi_0, \widehat B\phi_0\rangle
\end{equation}
where $\rho_m\rightarrow \infty$ and $\widehat B$ is the boundary operator with respect to $\tilde g$ as in \eqref{eq: operators}.

  Since we are only interested in the asymptotic behavior, by cutting off, we can assume that $\phi_0$ can be extended smoothly on the whole $\widetilde M$. Then outside a compact set, for $\overline \nabla=A\nabla A^{-1}$, we have
  $$
  \begin{cases}
    \overline \nabla_{e_a} \phi_0=-\frac{\sqrt{-1}}{2} \widetilde c(e_a)\phi_0,\\
    \overline \nabla _{\frac \partial {\partial\rho}}\phi_0=-\frac{\sqrt{-1}}{2}\frac{1}{u}\widetilde c(\frac\partial{\partial\rho}) \phi_0.
  \end{cases}
  $$

So
  $$
  \begin{cases}
    \widehat \nabla_{e_a} \phi_0=\widetilde \nabla_{e_a} \phi_0+\frac{\sqrt{-1}}{2} \widetilde c(e_a)\phi_0= (\widetilde \nabla_{e_a}- \overline \nabla_{e_a})\phi_0,\\
    \widehat \nabla _{\frac\partial{\partial\rho}}\phi_0=\widetilde \nabla_{\frac\partial{\partial\rho}} \phi_0+\frac{\sqrt{-1}}{2}\widetilde c(\frac\partial{\partial\rho}) \phi_0=(\widetilde \nabla_{\frac\partial{\partial\rho}}- u \overline \nabla _{\frac\partial{\partial\rho}})\phi_0.
  \end{cases}
  $$
  By the estimates in Lemma 2.1 of \cite{andersson1998scalar}, we have
  $$ |(\widetilde \nabla -\overline \nabla )\psi|\leq C|A^{-1}||\nabla'A||\psi | .$$
By Proposition \ref{prop: u}, $|A^{-1}||\nabla'A|=O(e^{-nk\rho})$. Also $|\phi_0|^2= O(e^{k\rho})$ by Proposition \ref{prop: spinor}, so
$|\widehat \nabla \phi_0|=O(e^{-(n-\frac{1}{2})k \rho})$. By Lemma \ref{lem: est}, the volume element of $(\widetilde M, \widetilde g)$ is of order $e^{(n-1)k\rho}$. We then have $\widehat \nabla \phi_0$, and therefore $\widehat D\phi_0$, are both in $L^2(\widetilde M, \widetilde g)$.

We now find $\phi_1\in W^{1,2}$ such that $\widehat D \phi_1=\widehat D \phi_0$ as follows.
Define a linear map on $W^{1,2}$ by
$l(\psi)=\int_{\widetilde M} \langle \widehat D \psi , \widehat D\phi_0\rangle$ and the sesquilinear form $B$ on $W^{1,2}$ by
$B(\psi , \varphi)=
\int_{\widetilde M} \langle \widehat D \psi, \widehat D \varphi\rangle$.
We claim that $B$ is bounded and coercive.

Let $\widetilde M_\rho$ be the region in $\widetilde M$ bounded by $\Sigma_\rho$ and let $\psi, \varphi\in C_c^\infty$.
On $\widetilde M_{\rho}\setminus \Omega$, $R=-n(n-1)$, so by the Lichnerowicz formula \eqref{eq: Lichnerowicz}, Proposition \ref{prop: bound} and the definition of $\widehat B$,
\begin{equation*}
\begin{split}
  \int_{\widetilde M_{\rho}\setminus \Omega}
(\langle\widehat \nabla \psi, \widehat \nabla \varphi\rangle-\langle\widehat D \psi, \widehat D \varphi\rangle)
=& \int_{\partial \Omega}
\langle \psi , (D^S +\frac{H}{2}
+ \sqrt{-1}c''(\nu))\varphi\rangle\\
&+ \int _{\Sigma_{\rho}}\langle \psi, (\nabla _\nu+c''(\nu) \widehat D)\varphi\rangle.
\end{split}
\end{equation*}
On $\Omega\subset \widetilde M_\rho$,
\begin{equation*}
  \begin{split}
 &    \int_{\Omega}
(\langle\widehat \nabla \psi, \widehat \nabla \varphi\rangle-\langle\widehat D \psi, \widehat D \varphi\rangle +\frac{1}{4}(R+n(n-1))\langle \psi, \varphi\rangle )\\
=& \int_{\partial \Omega}\langle \psi , -(D^S +\frac{H}{2} +\sqrt{-1}c''(\nu))\varphi\rangle.
  \end{split}
\end{equation*}
To be precise, $H$ in the two equations above are the mean curvatures of $\partial \Omega$ with respect to $g''$ and $g$ respectively, but since they agree (\eqref{eq: pde}, \eqref{eq: H}), so adding them up gives
$$
\int_{\widetilde M_{\rho}}
(\langle\widehat \nabla \psi, \widehat \nabla \varphi\rangle-\langle\widehat D \psi, \widehat D \varphi\rangle +\frac{1}{4}(R+n(n-1))\langle \psi, \varphi\rangle )
= \int_{\Sigma_\rho}
\langle \psi, (\nabla _\nu+c''(\nu) \widehat D)\varphi\rangle.
$$
As $R=-n(n-1)$ outside $\Omega$, so
\begin{equation*}
\begin{split}
B(\psi, \varphi)
= \int_{\widetilde M_{}} \langle \widehat D\psi, \widehat D\varphi \rangle
&= \int_{\widetilde M_{}} \left(\langle\widehat \nabla \psi,\widehat \nabla \varphi\rangle+ \frac{1}{4} (R+n(n-1)) \langle\psi ,\varphi\rangle \right)\\
&\leq C \|\psi\|_{W^{1,2}}\|\varphi\|_{W^{1,2}}.
\end{split}
\end{equation*}
So $B$ is bounded on $W^{1,2}$. On the other hand, as $R\geq -n(n-1)$, for $\psi\in C^{\infty}_c$,
\begin{equation*}
\begin{split}
\int_{\widetilde M} |\widehat D \psi|^2
\geq \int_{\widetilde M} |\widehat \nabla \psi|^2
&= \int_{\widetilde M} ( |\nabla \psi|^2 +\frac{n|\psi|^2}{4}  +\frac{\sqrt{-1}}{2}
(\langle D\psi, \psi\rangle -\langle \psi, D\psi\rangle ))\\
&= \int_{\widetilde M}(|\nabla \psi|^2 +\frac{n|\psi|^2}{4} ) \geq C\|\psi\|_{W^{1,2}}^2.
\end{split}
 \end{equation*}
So $B$ is also coercive. Then by Lax-Milgram theorem, there exists $\phi_1\in W^{1,2}$ such that $B(\phi_1, \psi)=l(\psi)$ for all $\psi \in W^{1,2}$. i.e.
$
\int_{\widetilde M} \langle \widehat D (\phi_1-\phi_0), \widehat \psi\rangle =0
$.
Let $\phi=\phi_1-\phi_0$ and define $\beta =\widehat D \phi$, so we have
$\int_{\widetilde M} \langle \beta, \widehat D\psi \rangle =0 $ for all $\psi \in W^{1,2}$.
This implies $\widehat D\beta =-n\sqrt{-1} \beta$ weakly, as ${\widehat D }^*=\widehat D+n\sqrt{-1}$.

As argued in \cite{shi2002positive} Lemma 3.3, $\beta \in W^{1,2}_{\textrm{loc}}$. Note also that in the weak sense, $\widehat D \beta =-n\sqrt{-1}\beta=-n\sqrt{-1}(\widehat D\phi_1-\widehat D\phi_0)\in L^2$.
Then
\begin{equation*}
\begin{split}
  \int_{\widetilde M_\rho}\langle \widehat  D\beta, \widehat D\beta \rangle
  &=\int_{\widetilde M_\rho}\langle (\widehat D+n\sqrt{-1})\widehat D \beta,  D\beta \rangle
  -\int_{\Sigma _\rho} \langle \widetilde c(\nu) \widehat D \beta, \widehat D\beta\rangle\\
  &=\int_{\widetilde M_\rho}\langle \widehat D (\widehat D+n\sqrt{-1})\beta,  \beta \rangle
  -\int_{\Sigma _\rho} \langle \widetilde c(\nu) \widehat D \beta, \widehat D\beta\rangle\\
  &=
  -\int_{\Sigma _\rho} \langle \widetilde c(\nu) \widehat D \beta, \widehat D\beta\rangle
  \leq \int_{\Sigma _\rho} |\widehat D \beta|^2.
  \end{split}
\end{equation*}
As $\displaystyle\int_{\widetilde M} |\widehat D \beta|^2<\infty$, there is a sequence $\rho_m\rightarrow \infty$ such that
$ \int _{\Sigma_{\rho_m}}|\widehat D \beta|^2\rightarrow 0$.
But then
$$\int_{\widetilde M}|\widehat D\beta |^2 =\lim_{m\rightarrow \infty} \int_{\widetilde M_{\rho_m}} |\widehat D\beta|^2
\leq \lim_{m \rightarrow \infty}  \int _{\Sigma_{\rho_m}}|\widehat D\beta|^2\rightarrow 0.$$
i.e. $\widehat D \beta=0$. As $\widehat D \beta =-n\sqrt{-1}\beta$, we have $\widehat D\phi=\beta =0$.
Now,
by the Lichnerowicz formula \eqref{eq: Lichnerowicz}, as $\widehat D\phi=0$,
\begin{equation*}
\begin{split}
  0&\leq
  \int_{\widetilde M_\rho} \left(|\widehat \nabla \phi|^2+\frac{1}{4} (R+n(n-1)) |\phi|^2\right)
  = \int_{\Sigma_\rho}\langle \widehat B \phi, \phi\rangle\\
  &=\int_{\Sigma_\rho}\langle \widehat B (\phi_1-\phi_0), \phi_1-\phi_0\rangle\\
  &=\int_{\Sigma_\rho}\langle \widehat B \phi_0, \phi_0\rangle
  +\left(\int_{\Sigma_\rho}\langle \widehat B \phi_1, \phi_1\rangle
  -\int_{\Sigma_\rho}\langle \widehat B \phi_0, \phi_1\rangle
  -\int_{\Sigma_\rho}\langle \widehat B \phi_1, \phi_0\rangle\right).
  \end{split}
\end{equation*}

We claim that there is $\rho_m\rightarrow \infty$ such that the three terms in the bracket above will tend to zero as $m\rightarrow \infty$. Consider
\begin{equation*}
\begin{split}
\int_{\Sigma_\rho}\langle \widehat B \phi_0, \phi_1\rangle
&=\int_{\Sigma_\rho}\langle (\widehat \nabla_\nu+ \widetilde c(\nu)\widehat D) \phi_0, \phi_1\rangle \\
&\leq (\int_{\Sigma_\rho}|\widehat \nabla_\nu \phi_0|^2)^{\frac{1}{2}}(\int_{\Sigma_\rho}|\phi_1|^2)^{\frac{1}{2}}+
(\int_{\Sigma_\rho}| \widehat D\phi_0|^2)^{\frac{1}{2}}(\int_{\Sigma_\rho}|\phi_1|^2)^{\frac{1}{2}}.
\end{split}
\end{equation*}
As $\displaystyle\int_{\widetilde M} |\widehat \nabla \phi_0|^2$, $\displaystyle\int_{\widetilde M} |\widehat D\phi_0|^2$ and $\displaystyle\int_{\widetilde M} | \phi_1|^2$ are all finite, there is $\rho_m\rightarrow \infty$ such that $\displaystyle\int_{\Sigma_{\rho_m}}\langle \widehat B \phi_0, \phi_1\rangle \rightarrow 0$.
Similarly, as $\displaystyle\int_{\widetilde M}|\widehat \phi _1|^2<\infty$, we can also assume that
$\displaystyle\int_{\Sigma_{\rho_m}}\langle \widehat B \phi_1, \phi_1\rangle \rightarrow 0$. \eqref{eq: killing harmonic} is proved.

Now, by Proposition \ref{prop: bound} and \ref{prop: D^S},
$$
0\leq \lim_{m\rightarrow \infty} \int_{\Sigma_{\rho_m}} \langle \widehat B \phi, \phi\rangle
=\lim _{m\rightarrow \infty} \int_{\Sigma_{\rho_m}} \langle \widehat B \phi_0, \phi_0\rangle
=\lim _{m\rightarrow \infty} \frac{1}{2}\int_{\Sigma_{\rho_m}} (H_0-H)|\phi_0|_{\widetilde g}^2.
$$
By Proposition \ref{prop: lim mass}, $\displaystyle \lim_{\rho\rightarrow\infty}\frac{1}{2}\int_{\Sigma_\rho}(H_0-H)|\phi_0|^2_{\widetilde g}$ exists, therefore
$$\lim_{\rho\rightarrow\infty}\frac{1}{2}\int_{\Sigma_\rho}(H_0-H)|\phi_0|^2_{\widetilde g}\geq0.$$

 To finish the proof, by Proposition \ref{prop: null}, we can let $\zeta=\zeta_a$. Let $\phi_{a, 0}'$ be the corresponding Killing spinor on $\mathbb{H}^n_{-k^2}$ and $\phi_{a,0}=A\phi'_{a,0}$ outside $\Omega$. By Proposition \ref{prop: spinor},  $|\phi_{a,0}|_{\tilde g}^2=|\phi'_{a,0}|_{g'}^2=-2kX\cdot \zeta_a$. So the above argument shows that
 $$-k \lim_{\rho\rightarrow \infty} \int_{\Sigma_\rho}(H_0-H)X\cdot \zeta_a\geq 0.$$
 In other words, $\displaystyle\lim_{\rho\rightarrow\infty} \int_{\Sigma_\rho}(H_0-H)X d\Sigma_\rho$ is a future-directed non-spacelike vector.
\end{proof}

\section{Positivity of the quasi-local mass of Shi-Tam}\label{sec: thm}
Now assume $n\geq 3$ and let $(\Omega , g)$ be as described in section \ref{sec: prelim}. Recall that  $B_0(R_1)$ and $B_0(R_2)$ are geodesic balls in $\mathbb{H}^n_{-k^2}$ such that $B_0(R_1)\subset \Omega_0\subset B_0(R_2)$. Our main result is the following
\begin{theorem}(cf. \cite{ShiTam2007} Theorem 3.1)\label{thm: m<0}
Let $n\geq 3$ and $(\Omega, g)$ be a compact spin $n$-manifold with smooth boundary $\Sigma$. Assuming the following conditions:
\begin{enumerate}
\item
The scalar curvature $R$ of $(\Omega, g)$ satisfies $R\ge -n(n-1) k^2$ for some $k>0$,
\item
$\Sigma$ is topologically a $(n-1)$-sphere with sectional curvature $K>-k^2$, mean curvature $H>0$ and $\Sigma$ can be isometrically embedded uniquely into $\mathbb{H}_{-k^2}^n$ with mean curvature $H_0$.
\end{enumerate}
Then for any future directed null vector $\zeta$ in $\mathbb{R}^{n, 1}$,
\begin{equation*}
m(\Omega, \zeta)=\int_{\Sigma}(H_0-H) W \cdot \zeta\leq 0
\end{equation*}
where
$W=(x_1, x_2, \cdots, x_n, \alpha t)$ with
$$1<\alpha=\coth kR_1+\frac{1}{\sinh kR_1}\left(\frac{\sinh^2 kR_2}{\sinh ^2 k R_1}-1 \right)^\frac{1}{2},$$
$X=(x_1, x_2, \cdots, x_n,  t)$ is the position vector in $\mathbb R^{n, 1}$ and the inner product is given by the Lorentz metric.
\end{theorem}
Let $(\phi_1, \cdots, \phi_n)$ denote the position vectors of points of $\mathbb{S}^{n-1}$ in $\mathbb{R}^n$. Let $\{\Sigma_\rho\}$ be the foliation of $\mathbb{H}_{-k^2}^n\setminus \Omega_0$ described in section \ref{sec: prelim}. We need the following:
\begin{lemma}\label{lem: phi}(cf. \cite{ShiTam2007} Lemma 3.3)
With the assumptions in Theorem \ref{thm: m<0}, let $(y_1, \cdots, y_n)\in \mathbb{R}^n$ such that
$\displaystyle\sum_{i=1}^n y_i^2=1$. Let $\displaystyle\phi=\sum_{i=1}^n \phi_i y_i$. Then for $\rho>0$,
\begin{enumerate}
  \item
  \begin{equation}\label{ineq: 1}
(\frac{\partial\phi}{\partial\rho})^2\leq (1-\phi^2)k^2 \sinh^{-2}kr \left( 1-(\frac{\partial r}{\partial \rho})^2\right).
\end{equation}
\item
\begin{equation}\label{ineq: 2}
  \frac{\partial r}{\partial \rho}\ge \frac{\sinh kR_1}{\sinh k R_2}.
\end{equation}
\end{enumerate}
Hence we have
\begin{equation}\label{eq: ineq}
\left|\frac{\partial \phi}{\partial \rho}\right|\leq \mu k \frac{\partial r}{\partial \rho}\textrm{\quad where
$\mu= \frac{1}{\sinh kR_1}\left(\frac{\sinh^2 kR_2}{\sinh ^2 k R_1}-1 \right)^\frac{1}{2}$.}
\end{equation}
\end{lemma}
\begin{proof}
 The position vectors in $\mathbb{R}^{n, 1}$ can be parametrized by
$$X=\frac{1}{k}(\sinh kr \cos \theta, \sinh kr \sin \theta \vec z, \cosh kr)$$
where $\vec z\in \mathbb{S}^{n-2}\subset \mathbb{R}^{n-1}$. Then the hyperbolic metric (outside $\Omega_0$) is
$$g'=d\rho^2+ g_\rho=dr^2+ k^{-2} \sinh^2 r (d\theta^2+\sin ^2 \theta d\sigma)$$
where $d\sigma $ is the standard metric on $\mathbb{S}^{n-2}$.
Compute $g'(\frac{\partial }{\partial \rho}, \frac{\partial }{\partial \rho})$ in the above two forms of $g'$, we have
\begin{equation*}
  \begin{split}
1&=(\frac{\partial r}{\partial \rho})^2+k^{-2} \sinh ^2 kr\left((\frac{\partial \theta}{\partial \rho})^2+\sin^2 \theta d\sigma (\frac{\partial }{\partial \rho}, \frac{\partial }{\partial \rho})\right)\\
&\geq
(\frac{\partial r}{\partial \rho})^2+k^{-2} \sinh ^2 kr(\frac{\partial \theta}{\partial \rho})^2.
  \end{split}
\end{equation*}
Since $\phi=\cos \theta$, \eqref{ineq: 1} follows by multiplying $\sin ^2 \theta$ to the above.

Recall that $o\in \Omega_0$ and $r$ is the geodesic distance from $o$.
Let $p\in \Sigma $ and let $\gamma$ be the (arc-length parametrized) outward pointing geodesic through $p$ which is orthogonal to $\Sigma$.
Let $q$ be the point on $\gamma$ such that $a=d(o, q)=d(o, \gamma)$. We will simple denote by $yz$ to be the distance $d(y,z)$ between two points $y$ and $z$. Then for $x=\gamma (\rho)$, $\rho>0$, cosine law gives $\cosh (k\,r)=\cosh( k\,a) \cosh( k\cdot qx)$, where $r=ox$. By \cite{do1970rigidity}, the region bounded by $\Sigma_0$ is geodesically convex, so we have $qx=qp+\rho$. Thus along $\gamma$ we have $$\frac{\partial r}{\partial \rho}= \frac{\cosh (k\,a )\sinh (k \cdot qx)}{ \sinh ( k \,r)}>0.$$
Also, as $\frac{\sinh (k\cdot qx)}{\sinh (k \,r)}=\sin (\angle xoq)$ if $o\ne q$, we see that $\frac{\partial r}{\partial \rho}$ is non-decreasing along $\gamma$ for $\rho>0$, which is also true if $o=q$. We now estimate $\frac{\partial r}{\partial \rho}(p)$. To show \eqref{ineq: 2}, we can assume $o\ne q$, for otherwise $\frac{\partial r}{\partial \rho}=1\ge \frac{\sinh kR_1}{\sinh kR_2}$.

Let $\beta$ be the geodesic on the hyperbolic plane  containing the triangle $\triangle oqp$ such that it is perpendicular to $\gamma$ at $p$. Suppose $\beta$ meets $\partial B_0(R_2)$ at $y$ and $z$. Then $\frac{\partial r}{\partial \rho}=\sin (\angle opy)$. By sine law,
$$\frac{\partial r}{\partial \rho}= \sin (\angle opy)= \frac{\sinh k R_2}{\sinh (k \cdot op)}\sin (\angle oyp)\ge \sin (\angle oyp)\ge \frac{\sinh kR_1}{\sinh kR_2}.$$
This proves \eqref{ineq: 2}.
The inequality \eqref{eq: ineq} follows from \eqref{ineq: 1} and \eqref{ineq: 2}, together with $|\phi|\le 1$.
\end{proof}
We are now ready to prove our main result.
\begin{proof}[Proof of Theorem \ref{thm: m<0}]
$X$ can be expressed as
\begin{equation*}
\begin{split}
X&=\frac{1}{k}(\sinh (kr)Y, \cosh kr)=\frac{1}{k}(\sinh (kr) y_1, \cdots, \sinh (kr) y_n, \cosh kr).
\end{split}
\end{equation*}
where $\displaystyle|Y|^2=\sum_{i=1}^n y_i^2=1.$
Without loss of generality we can assume that
$\zeta=(\zeta_1,\cdots, \zeta_n, 1)$ where $\displaystyle\sum_{i=1}^n \zeta_i^2=1$.

Let $\displaystyle \phi=\sum_{i=1}^n y_i \zeta_i$, then Lemma \ref{lem: der} implies (we omit $d\Sigma_\rho$ for convenience)
\begin{equation}\label{eq: W}
\begin{split}
&
\frac{d}{d\rho}(\int_{\Sigma_\rho}(H_0-H)W(\rho, p)\cdot \zeta)\\
=&
-\int_{\Sigma_\rho}u^{-1}(u-1)^2  (\frac{1}{2}(R^\rho+(n-1)(n-2)k^2)(\phi\sinh kr -\alpha \cosh kr)  \\
&+ H_0\frac{\partial }{\partial \rho}(\phi\sinh kr -\alpha \cosh kr))\\
=&
-\int_{\Sigma_\rho}u^{-1}(u-1)^2  (\frac{1}{2}(H_0^2-|A|^2)(\phi\sinh kr -\alpha \cosh kr) \\
 &+ H_0\frac{\partial }{\partial \rho}(\phi\sinh kr -\alpha \cosh kr))\\
=&
-\int_{\Sigma_\rho}u^{-1}(u-1)^2  B \quad\text{where}
\end{split}
\end{equation}
\begin{equation}\label{eq: B}
\begin{split}
B
=&\frac{1}{2}(H_0^2-|A|^2)(\phi \sinh kr -\alpha \cosh kr)\\
&+kH_0 (\phi \cosh kr \frac{\partial r}{\partial \rho}+\frac{1}{k} \sinh kr \frac {\partial \phi}{\partial \rho}- \alpha \sinh kr \frac{\partial r}{\partial \rho}).
\end{split}
\end{equation}

Here $A$ is the second fundamental form of $\Sigma_\rho$ with respect to the hyperbolic metric. Let $\lambda_a(p, \rho)$ be the principal curvature of $\Sigma_\rho$. Then $\lambda_a =k \tanh k(\mu_a+\rho)$, $k$, or $k \coth k(\mu_a+\rho )$ with $\mu_a>0$ (\cite{WY} p.255). In particular,
\begin{equation*}
H_0^2-|A|^2=2\sum_{a<b}\lambda_a \lambda_b\geq 0.
\end{equation*}

We want to show $B\leq 0$. For the first term of B, consider
\begin{equation}\label{eq: pos I}
\begin{split}
\phi \sinh kr -\alpha \cosh kr\leq \sinh kr -\cosh kr<0.
\end{split}
\end{equation}
To show that the last term of R.H.S. of \eqref{eq: B} is also negative, it suffices to show
 $$\phi \cosh kr \frac{\partial r}{\partial \rho}+\frac{1}{k}\sinh kr \frac{\partial \phi}{\partial \rho}-\alpha \sinh kr \frac{\partial r}{\partial \rho}<0.$$
Indeed, by \eqref{eq: ineq}, we have $\frac{\partial r}{\partial \rho}>0$ and
\begin{equation*}
\begin{split}
&\phi \cosh kr \frac{\partial r}{\partial \rho}+\frac{1}{k}\sinh kr \frac{\partial \phi}{\partial \rho}-\alpha \sinh kr \frac{\partial r}{\partial \rho}\\
\leq& \cosh kr \frac{\partial r}{\partial \rho} +\frac{1}{k}\sinh kr (\mu k \frac{\partial r}{\partial \rho})-\alpha \sinh kr \frac{\partial r}{\partial \rho}\\
=&(\cosh kr +\sinh kr(\mu -\alpha))\frac{\partial r}{\partial \rho}\\
=&(\cosh kr -\sinh kr \coth kR_1)\frac{\partial r}{\partial \rho}
<0 \quad \text{\quad as $r>R_1$.}
\end{split}
\end{equation*}
Substituting into \eqref{eq: W}, we have
$$\frac{d}{d\rho}(\int_{\Sigma_\rho}(H_0-H)W(\rho, p)\cdot \zeta)\geq 0.$$
By Theorem \ref{thm: PMT} and Corollary \ref{corollary}, we conclude that
$$m(\Omega, \zeta)=\int_{\Sigma}(H_0-H) W \cdot \zeta\leq 0.$$
The proof is completed.
\end{proof}

\begin{remark}
  As remarked in \cite{ShiTam2007} Example 3.11-3.13, in some situations, $\alpha$ in Theorem \ref{thm: m<0} can be chosen to be $1$. Indeed, it is not hard to see that this is possible in the following cases:
  \begin{enumerate}
    \item
    $\Sigma$ is isometric to a standard sphere and $H$ is constant, or
    \item
    $\Sigma$ is isometric to a standard sphere and $H$ is orthogonal to the first eigenfunctions of $\mathbb S^{n-1}$ (i.e. restriction of the position functions of $\mathbb S^{n-1}\subset \mathbb R^n$). Here we assume that $0=(0, \cdots, 0, 1/k)$ is the center of the geodesic ball in $\mathbb H^n_{-k^2}$ enclosed by $\Sigma_0$.
  \end{enumerate}
\end{remark}

\end{document}